\documentclass[11pt]{article}
\usepackage[]{amsmath,amssymb}
\usepackage{amscd}
\usepackage{cite}
\newtheorem{theorem}{Theorem}[section]
\newtheorem{lemma}[theorem]{Lemma}
\newtheorem{proposition}[theorem]{Proposition}

\newtheorem{exAux}[theorem]{Example}

\newtheorem{Def}[theorem]{Definition}
\newenvironment{definition}{\begin{Def} \rm}{\end{Def}}
\newtheorem{Note}[theorem]{Note}
\newenvironment{note}{\begin{Note} \rm}{\end{Note}}
\newtheorem{Problem}[theorem]{Problem}

\newtheorem{Rem}[theorem]{Remark}

\newtheorem{Not}[theorem]{Notation}

\newtheorem{Conj}[theorem]{Conjecture}
\newenvironment{conjecture}{\begin{Conj}}{\end{Conj}}
\newtheorem{Ass}[theorem]{Assumption}

\newenvironment{proof}{\medskip\noindent{\bf Proof.\ }}{\qed\medskip}

\newcommand{\qed}{\hfill\mbox{$\Box$\qquad\qquad}}

\renewcommand{\th}{\theta}

\newcommand{\K}{\mathbb{K}}
\newcommand{\F}{\mathbb{F}}

\renewcommand{\indent}{\hspace{6mm}}

\begin{document}
\thispagestyle{empty}

\begin{center}
\LARGE \bf
Tridiagonal pairs of $q$-Racah type\\
and the $\mu$-conjecture
\end{center}

\smallskip

\begin{center}
\Large
Kazumasa Nomura and 
Paul Terwilliger
\end{center}

\smallskip

\begin{quote}
\small 
\begin{center}
\bf Abstract
\end{center}
\indent
Let $\K$ denote a field and let $V$ denote a vector space over $\K$ with 
finite positive dimension.
We consider a pair of linear transformations $A:V \to V$
and $A^*:V \to V$ that satisfy the following conditions:
(i)
each of $A,A^*$ is diagonalizable;
(ii)
there exists an ordering $\lbrace V_i\rbrace_{i=0}^d$ of the eigenspaces of 
$A$ such that
$A^* V_i \subseteq V_{i-1} + V_{i} + V_{i+1}$ for $0 \leq i \leq d$,
where $V_{-1}=0$ and $V_{d+1}=0$;
(iii)
there exists an ordering $\lbrace V^*_i\rbrace_{i=0}^\delta$ 
of the eigenspaces of $A^*$ such that
$A V^*_i \subseteq V^*_{i-1} + V^*_{i} + V^*_{i+1}$ for
 $0 \leq i \leq \delta$,
where $V^*_{-1}=0$ and $V^*_{\delta+1}=0$;
(iv) 
there is no subspace $W$ of $V$ such that
$AW \subseteq W$, $A^* W \subseteq W$, $W \neq 0$, $W \neq V$.
We call such a pair a {\it tridiagonal pair} on $V$.
It is known that $d=\delta$ and for $0 \leq i \leq d$
the dimensions of $V_i$, $V_{d-i}$, $V^*_i$, $V^*_{d-i}$ coincide.
We say the pair $A,A^*$ is {\it sharp} whenever $\dim V_0=1$.
It is known that if 
$\K$ is algebraically closed then $A,A^*$ is sharp.
A conjectured classification of the sharp tridiagonal pairs
was recently introduced by T. Ito and the second author.
Shortly afterwards we introduced a conjecture, called the 
{\em $\mu$-conjecture},
which implies the classification conjecture.
In this paper we show that the $\mu$-conjecture holds
in a special case called $q$-Racah.


\end{quote}

\section{Tridiagonal pairs}

\indent
Throughout the paper $\K$ denotes a field.

\medskip

We begin by recalling the notion of a tridiagonal pair. 
We will use the following terms.
Let $V$ denote a vector space over $\K$ with finite positive dimension.
For a linear transformation $A:V\to V$ and a subspace $W \subseteq V$,
we call $W$ an {\em eigenspace} of $A$ whenever $W\not=0$ and there exists 
$\th \in \K$ such that $W=\{v \in V \,|\, Av = \th v\}$;
in this case $\th$ is the {\em eigenvalue} of $A$ associated with $W$.
We say that $A$ is {\em diagonalizable} whenever $V$ is spanned by the 
eigenspaces of $A$.

\begin{definition}  {\rm \cite[Definition 1.1]{TD00}} \label{def:tdp}  \samepage
Let $V$ denote a vector space over $\K$ with finite positive dimension. 
By a {\em tridiagonal pair} on $V$ we mean an ordered pair of linear 
transformations $A:V \to V$ and $A^*:V \to V$ that satisfy the following 
four conditions.
\begin{enumerate}
\item 
Each of $A,A^*$ is diagonalizable.
\item 
There exists an ordering $\{V_i\}_{i=0}^d$ of the eigenspaces of $A$ 
such that 
\begin{equation}               \label{eq:t1}
A^* V_i \subseteq V_{i-1} + V_i+ V_{i+1} \qquad \qquad (0 \leq i \leq d),
\end{equation}
where $V_{-1} = 0$ and $V_{d+1}= 0$.
\item
There exists an ordering $\{V^*_i\}_{i=0}^{\delta}$ of the eigenspaces of 
$A^*$ such that 
\begin{equation}                \label{eq:t2}
A V^*_i \subseteq V^*_{i-1} + V^*_i+ V^*_{i+1} 
\qquad \qquad (0 \leq i \leq \delta),
\end{equation}
where $V^*_{-1} = 0$ and $V^*_{\delta+1}= 0$.
\item 
There does not exist a subspace $W$ of $V$ such  that $AW\subseteq W$,
$A^*W\subseteq W$, $W\not=0$, $W\not=V$.
\end{enumerate}
We say that the pair $A,A^*$ is {\em over} $\K$.
\end{definition}

\begin{note}   \label{note:star}        \samepage
According to a common notational convention $A^*$ denotes 
the conjugate-transpose of $A$. We are not using this convention.
In a tridiagonal pair $A,A^*$ the linear transformations $A$ and $A^*$
are arbitrary subject to (i)--(iv) above.
\end{note}

\medskip

We refer the reader to 
\cite{TD00,shape,nomsharp,nomtowards,nomstructure,NT:muconj,
LS99,qrac,madrid,tdanduq,NN,qtet,Ev,IT:Krawt,IT:qRacah} for background on tridiagonal pairs.

\medskip

In order to motivate our results we recall some facts about tridiagonal pairs.
Let $A,A^*$ denote a tridiagonal pair on $V$, as in Definition \ref{def:tdp}. 
By \cite[Lemma 4.5]{TD00} the integers $d$ and $\delta$ from (ii), (iii) are 
equal; we call this common value the {\em diameter} of the pair.
An ordering of the eigenspaces of $A$ (resp. $A^*$)
is said to be {\em standard} whenever it satisfies
\eqref{eq:t1} (resp. \eqref{eq:t2}). We comment on the uniqueness of the
standard ordering. Let $\{V_i\}_{i=0}^d$ denote a standard
ordering of the eigenspaces of $A$. By \cite[Lemma 2.4]{TD00}, the ordering
$\{V_{d-i}\}_{i=0}^d$ is also standard and no further ordering
is standard. A similar result holds for the eigenspaces
of $A^*$. Let $\{V_i\}_{i=0}^d$ (resp. $\{V^*_i\}_{i=0}^d$) 
denote a standard ordering of the eigenspaces of $A$ (resp. $A^*$).
By \cite[Corollary 5.7]{TD00}, for $0 \leq i \leq d$ the spaces $V_i$, $V^*_i$
have the same dimension; we denote this common dimension by $\rho_i$. 
By \cite[Corollaries 5.7, 6.6]{TD00} the sequence $\{\rho_i\}_{i=0}^d$ is 
symmetric and unimodal;
that is $\rho_i=\rho_{d-i}$ for $0 \leq i \leq d$ and
$\rho_{i-1} \leq \rho_i$ for $1 \leq i \leq d/2$.
We call the sequence $\{\rho_i\}_{i=0}^d$ the {\em shape} of $A,A^*$.
We say $A,A^*$ is {\it sharp} whenever $\rho_0=1$.
If $\K$ is algebraically closed then
$A,A^*$ is sharp \cite[Theorem~1.3]{nomstructure}. 

\medskip

We now summarize the present paper.
A conjectured classification of the sharp tridiagonal pairs was introduced 
in \cite[Conjecture 14.6]{IT:Krawt} and studied carefully in 
\cite{nomsharp,nomtowards,nomstructure}; see Conjecture \ref{conj:main} below.
Shortly afterwards we introduced a conjecture, called the {\em $\mu$-conjecture},
which implies the classification conjecture. The $\mu$-conjecture is roughly
described as follows. We start with a sequence
$p=(\{\th_i\}_{i=0}^d;\{\th^*_i\}_{i=0}^d)$ of scalars taken from $\K$ that
satisfy the known constraints on the eigenvalues of a tridiagonal pair over 
$\K$ of diameter $d$; these are conditions (i), (ii) of Conjecture 
\ref{conj:main}.
Following \cite[Definition 2.4]{nomstructure} we associate with $p$ 
an associative $\K$-algebra $T$ defined by generators
and relations; see Definition \ref{def:T} below.
We are interested in the $\K$-algebra $e^*_0Te^*_0$ where $e^*_0$ is a certain
idempotent element of $T$. Let $\{x_i\}_{i=1}^d$ denote mutually commuting
indeterminates. Let $\K[x_1,\ldots,x_d]$ denote the $\K$-algebra consisting of
the polynomials in $\{x_i\}_{i=1}^d$ that have all coefficients in $\K$.
In \cite[Corollary 6.3]{NT:muconj} we displayed a surjective 
$\K$-algebra homomorphism $\mu: \K[x_1,\ldots,x_d] \to e^*_0Te^*_0$.
The {\em $\mu$-conjecture} \cite[Conjecture 6.4]{NT:muconj} asserts that
$\mu$ is an isomorphism.
We have shown that the $\mu$-conjecture implies the classification 
conjecture \cite[Theorem 10.1]{NT:muconj}
and that the $\mu$-conjecture holds for $d \leq 5$ 
\cite[Theorem 12.1]{NT:muconj}.
In the present paper we obtain the following additional evidence that the
$\mu$-conjecture is true.
There is a general class of parameters $p$ said to have
{\em $q$-Racah type} \cite[Definition 3.1]{IT:qRacah}. 
In \cite[Theorem 3.3]{IT:qRacah} T. Ito and the second author verified the
classification conjecture for the case in which $p$ has $q$-Racah type
and $\K$ is algebraically closed; see Proposition \ref{prop:ITqRacah} below.
Making heavy use of this result, we verify the $\mu$-conjecture for the case 
in which $p$ has $q$-Racah type, with no restriction on $\K$. 
Our main result is Theorem \ref{thm:main}.
On our way to Theorem \ref{thm:main} we obtain two related results 
Theorem \ref{thm:A} and \ref{thm:FK}, which might be of independent interest.

\section{Tridiagonal systems}

\indent
When working with a tridiagonal pair, it is often convenient to consider
a closely related object called a tridiagonal system.
To define a tridiagonal system, we recall a few concepts from linear algebra.
Let $V$ denote a vector space over $\K$ with finite positive dimension.
Let ${\rm End}_{\K}(V)$ denote the $\K$-algebra of all linear
transformations from $V$ to $V$.
Let $A$ denote a diagonalizable element of $\text{End}_{\K}(V)$.
Let $\{V_i\}_{i=0}^d$ denote an ordering of the eigenspaces of $A$
and let $\{\th_i\}_{i=0}^d$ denote the corresponding ordering of the 
eigenvalues of $A$.
For $0 \leq i \leq d$ define $E_i \in \text{End}_{\K}(V)$ such that 
$(E_i-I)V_i=0$ and $E_iV_j=0$ for $j \neq i$ $(0 \leq j \leq d)$.
Here $I$ denotes the identity of $\mbox{\rm End}_{\K}(V)$.
We call $E_i$ the {\em primitive idempotent} of $A$ corresponding to $V_i$
(or $\th_i$).
Observe that
(i) $I=\sum_{i=0}^d E_i$;
(ii) $E_iE_j=\delta_{i,j}E_i$ $(0 \leq i,j \leq d)$;
(iii) $V_i=E_iV$ $(0 \leq i \leq d)$;
(iv) $A=\sum_{i=0}^d \theta_i E_i$.
Moreover
\begin{equation}         \label{eq:defEi}
  E_i=\prod_{\stackrel{0 \leq j \leq d}{j \neq i}}
          \frac{A-\theta_jI}{\theta_i-\theta_j}.
\end{equation}

\medskip

Now let $A,A^*$ denote a tridiagonal pair on $V$.
An ordering of the primitive idempotents or eigenvalues of $A$ (resp. $A^*$)
is said to be {\em standard} whenever the corresponding ordering of the 
eigenspaces of $A$ (resp. $A^*$) is standard.

\medskip

\begin{definition} \cite[Definition 2.1]{TD00} \label{def:TDsystem}  \samepage
Let $V$ denote a vector space over $\K$ with finite positive dimension.
By a {\em tridiagonal system} on $V$ we mean a sequence
\[
 \Phi=(A;\{E_i\}_{i=0}^d;A^*;\{E^*_i\}_{i=0}^d)
\]
that satisfies (i)--(iii) below.
\begin{itemize}
\item[(i)]
$A,A^*$ is a tridiagonal pair on $V$.
\item[(ii)]
$\{E_i\}_{i=0}^d$ is a standard ordering
of the primitive idempotents of $A$.
\item[(iii)]
$\{E^*_i\}_{i=0}^d$ is a standard ordering
of the primitive idempotents of $A^*$.
\end{itemize}
We say $\Phi$ is {\em over} $\K$. 
We call $V$ the {\em vector space underlying $\Phi$}.
\end{definition}

\medskip

The notion of isomorphism for tridiagonal systems
is defined in \cite[Definition 3.1]{nomsharp}.
The following result is immediate from lines \eqref{eq:t1}, \eqref{eq:t2}
and Definition \ref{def:TDsystem}.

\medskip

\begin{lemma} {\rm \cite[Lemma 2.5]{nomtowards}}  \label{lem:trid}  \samepage
Let $(A;\{E_i\}_{i=0}^d;A^*;\{E^*_i\}_{i=0}^d)$ denote a tridiagonal system.
Then for $0 \leq i,j,k \leq d$ the following {\rm (i)}, {\rm (ii)} hold.
\begin{itemize}
\item[\rm (i)]
$E_i {A^*}^k E_j = 0\;\;$ if  $k<|i-j|$.
\item[\rm (ii)]
$E^*_i A^k E^*_j =0\;\;$ if $k<|i-j|$.
\end{itemize}
\end{lemma}

\begin{definition}        \label{def}        \samepage
Let $\Phi=(A;\{E_i\}_{i=0}^d;A^*;\{E^*_i\}_{i=0}^d)$ denote a tridiagonal 
system on $V$.
For $0 \leq i \leq d$ let $\th_i$ (resp. $\th^*_i$) denote the eigenvalue of 
$A$ (resp. $A^*$) associated with the eigenspace $E_iV$ (resp. $E^*_iV$).
We call $\{\th_i\}_{i=0}^d$ (resp. $\{\th^*_i\}_{i=0}^d$)
the {\em eigenvalue sequence} (resp. {\em dual eigenvalue sequence}) of $\Phi$.
We observe that $\{\th_i\}_{i=0}^d$ (resp. $\{\th^*_i\}_{i=0}^d$) are mutually 
distinct and contained in $\K$. 
We say $\Phi$ is {\it sharp} whenever the tridiagonal pair $A,A^*$ is sharp.
\end{definition}

\medskip

Let $\Phi$ denote a tridiagonal system over $\K$ with
eigenvalue sequence $\{\th_i\}_{i=0}^d$ and dual eigenvalue
sequence $\{\th^*_i\}_{i=0}^d$. 
By \cite[Theorem 11.1]{TD00} the expressions
\begin{equation}      \label{eq:indep0}
 \frac{\th_{i-2}-\th_{i+1}}{\th_{i-1}-\th_i},
\qquad\qquad
 \frac{\th^*_{i-2}-\th^*_{i+1}}{\th^*_{i-1}-\th^*_i}
\end{equation}
are equal and independent of $i$ for $2 \leq i \leq d-1$.
For this constraint the ``most general'' solution is
\begin{align}
 & \th_i= 
   a + bq^{2i-d}+cq^{d-2i} \qquad\qquad (0 \leq i \leq d), \label{eq:1} \\
 & \th^*_i =
   a^*+b^*q^{2i-d}+c^*q^{d-2i} \qquad\qquad (0 \leq i \leq d), \label{eq:2} \\
 & q, \; a, \; b, \; c,\; a^*, \; b^*, \; c^* \in \overline{\K}, \label{eq:3}\\
 & q \neq 0, \qquad q^2 \neq 1, \qquad q^2 \neq -1, \qquad 
   bb^*cc^* \neq 0,                                            \label{eq:4}
\end{align}
where $\overline{\K}$ denotes the algebraic closure of $\K$.
For this solution $q^2+q^{-2}+1$ is the common value of \eqref{eq:indep0}.
The tridiagonal system $\Phi$ is said to have {\em q-Racah type}
whenever \eqref{eq:1}--\eqref{eq:4} hold. 
The following definition is more general.

\medskip

\begin{definition}             \label{def:qRacah}
Let $d$ denote a nonnegative integer and let
$(\{\th_i\}_{i=0}^d; \{\th^*_i\}_{i=0}^d)$ denote a sequence
of scalars taken from $\K$.
We call this sequence {\em $q$-Racah} whenever the following
(i), (ii) hold.
\begin{itemize}
\item[(i)]
$\th_i \neq \th_j$, $\th^*_i \neq \th^*_j$ if $i \neq j$
$(0 \leq i,j \leq d)$.
\item[(ii)]
There exist $q,a,b,c,a^*,b^*,c^*$ that satisfy \eqref{eq:1}--\eqref{eq:4}.
\end{itemize}
\end{definition}

\medskip

We will return to the subject of $q$-Racah a bit later.
We now recall the split sequence of a sharp tridiagonal system.
We will use the following notation.

\medskip

\begin{definition}  \label{def:tau}    \samepage
Let $\lambda$ denote an indeterminate and let $\K[\lambda]$ denote the 
$\K$-algebra consisting of the polynomials in $\lambda$ that have 
all coefficients in $\K$.
Let $d$ denote a nonnegative integer and let
$(\{\th_i\}_{i=0}^d; \{\th^*_i\}_{i=0}^d)$ denote a sequence of scalars 
taken from $\K$.
Then for $0 \leq i \leq d$ we define the following polynomials in $\K[\lambda]$:
\begin{align*}
 \tau_i &= (\lambda-\th_0)(\lambda-\th_1)\cdots(\lambda -\th_{i-1}), \\
 \eta_i &= (\lambda-\th_d)(\lambda-\th_{d-1})\cdots(\lambda-\th_{d-i+1}),  \\
 \tau^*_i &= (\lambda-\th^*_0)(\lambda-\th^*_1)\cdots(\lambda-\th^*_{i-1}), \\
 \eta^*_i &= (\lambda-\th^*_d)(\lambda-\th^*_{d-1})\cdots(\lambda-\th^*_{d-i+1}).
\end{align*}
Note that each of $\tau_i$, $\eta_i$, $\tau^*_i$, $\eta^*_i$ is monic with 
degree $i$.
\end{definition}

\begin{definition} {\rm\cite[Definition 2.5]{NT:muconj}} \label{def:split} \samepage
Let $(A; \{E_i\}_{i=0}^d; A^*; \{E^*_i\}_{i=0}^d)$ denote a sharp tridiagonal 
system over $\K$, with eigenvalue sequence $\{\th_i\}_{i=0}^d$
and dual eigenvalue sequence $\{\th^*_i\}_{i=0}^d$.
By \cite[Lemma 5.4]{nomstructure},
for $0 \leq i \leq d$ there exists a unique $\zeta_i \in \K$ such that 
\[
E^*_0 \tau_i(A) E^*_0 = 
\frac{\zeta_i E^*_0}
{(\theta^*_0-\theta^*_1)(\theta^*_0-\theta^*_2)\cdots(\theta^*_0-\theta^*_i)}. 
\]
Note that $\zeta_0=1$.
We call $\lbrace \zeta_i \rbrace_{i=0}^d$ the {\em split sequence} of the 
tridiagonal system.
\end{definition}

\begin{definition} \cite[Definition 6.2]{nomsharp}   \samepage
Let $\Phi$ denote a sharp tridiagonal system.
By the {\em parameter array} of $\Phi$ we mean the sequence
 $(\{\theta_i\}_{i=0}^d; \{\theta^*_i\}_{i=0}^d; \{\zeta_i\}_{i=0}^d)$
where  $\{\theta_i\}_{i=0}^d$ (resp. $\{\theta^*_i\}_{i=0}^d$) is the 
eigenvalue sequence (resp. dual eigenvalue sequence) of $\Phi$ and
$\{\zeta_i\}_{i=0}^d$ is the split sequence of $\Phi$.
\end{definition}

\section{The classification conjecture}

\indent
To motivate our results we recall a conjectured
classification of the tridiagonal systems due to 
T. Ito and the second author.

\medskip

\begin{conjecture} {\rm \cite[Conjecture~14.6]{IT:Krawt}}\label{conj:main}
\samepage  
Let $d$ denote a nonnegative integer and let
\begin{equation}         \label{eq:parray}
 (\{\th_i\}_{i=0}^d; \{\th^*_i\}_{i=0}^d; \{\zeta_i\}_{i=0}^d)
\end{equation}
denote a sequence of scalars taken from $\K$.
Then there exists a sharp tridiagonal system $\Phi$ over $\K$ with parameter 
array \eqref{eq:parray} if and only if {\rm (i)--(iii)} hold below.
\begin{itemize}
\item[\rm (i)]
$\th_i \neq \th_j$, $\th^*_i \neq \th^*_j$ if $i \neq j$
$(0 \leq i,j \leq d)$.
\item[\rm (ii)]
The expressions
\begin{equation}    \label{eq:indep}
 \frac{\th_{i-2}-\th_{i+1}}{\th_{i-1}-\th_{i}},  \qquad\qquad
 \frac{\th^*_{i-2}-\th^*_{i+1}}{\th^*_{i-1}-\th^*_{i}}
\end{equation}
are equal and independent of $i$ for $2 \leq i \leq d-1$.
\item[\rm (iii)]
$\zeta_0=1$, $\zeta_d \neq 0$ and
\begin{equation}  \label{eq:ineq}
    0 \neq  \sum_{i=0}^d \eta_{d-i}(\th_0)\eta^*_{d-i}(\th^*_0) \zeta_i.
\end{equation}
\end{itemize}
Suppose {\rm (i)--(iii)} hold.
Then $\Phi$ is unique up to isomorphism of tridiagonal systems.
\end{conjecture}

\medskip

In \cite[Section 8]{nomsharp} we proved the ``only if'' direction
of Conjecture \ref{conj:main}.
In \cite[Theorem 1.6]{nomstructure} we proved the last assertion
of Conjecture \ref{conj:main}. 
Concerning the ``if'' direction of Conjecture \ref{conj:main},
we proved this for $d\leq 5$ in \cite[Corollary 12.2]{NT:muconj}. 
We also have the following result due to T. Ito and the second author.

\medskip

\begin{proposition}  {\rm \cite[Theorem 3.3]{IT:qRacah}} \label{prop:ITqRacah}
\samepage
Assume $\K$ is algebraically closed.
Let $d$ denote a nonnegative integer and let
$(\{\th_i\}_{i=0}^d; \{\th^*_i\}_{i=0}^d)$ denote a sequence of scalars taken
from $\K$ that is $q$-Racah in the sense of Definition {\rm \ref{def:qRacah}}.
Let $\{\zeta_i\}_{i=0}^d$ denote a sequence of scalars taken from $\K$
that satisfies condition {\rm (iii)} of Conjecture {\rm \ref{conj:main}}.
Then there exists a sharp tridiagonal system over $\K$ that has
parameter array
$(\{\th_i\}_{i=0}^d; \{\th^*_i\}_{i=0}^d; \{\zeta_i\}_{i=0}^d)$.
\end{proposition}

\section{The $\mu$-conjecture}

\indent
In this section we recall the $\mu$-conjecture.
It has to do with the following algebra.

\medskip

\begin{definition}{\rm \cite[Definition 2.4]{nomstructure}}\label{def:T}
\samepage
Let $d$ denote a nonnegative integer, and let
$p=(\{\th_i\}_{i=0}^d; \{\th^*_i\}_{i=0}^d)$ denote a sequence of scalars taken 
from $\K$ that satisfies conditions (i), (ii) of Conjecture {\rm \ref{conj:main}}.
Let $T=T(p,{\K})$ denote the associative 
$\K$-algebra with $1$, defined by generators
$a$, $\{e_i\}_{i=0}^d$, $a^*$, $\{e^*_i\}_{i=0}^d$ and relations
\begin{equation}                            \label{eq:eiej}
  e_ie_j=\delta_{i,j}e_i, \qquad 
  e^*_ie^*_j=\delta_{i,j}e^*_i \qquad\qquad 
  (0 \leq i,j \leq d),
\end{equation}
\begin{equation}                            \label{eq:sumei}
  1=\sum_{i=0}^d e_i, \qquad\qquad
  1=\sum_{i=0}^d e^*_i,
\end{equation}
\begin{equation}                                  \label{eq:sumthiei}
   a = \sum_{i=0}^d \theta_ie_i, \qquad\qquad
   a^* = \sum_{i=0}^d \theta^*_i e^*_i,
\end{equation}
\begin{equation}                                     \label{eq:esiakesj}
 e^*_i a^k e^*_j = 0  \qquad \text{if $\;k<|i-j|$}
                  \qquad\qquad (0 \leq i,j,k \leq d),
\end{equation}
\begin{equation}                                      \label{eq:eiaskej}
 e_i {a^*}^k e_j = 0  \qquad \text{if $\;k<|i-j|$} 
                  \qquad\qquad (0 \leq i,j,k \leq d). 
\end{equation}
\end{definition}

\medskip

The algebra $T$ is related to tridiagonal systems as follows.

\medskip

\begin{lemma}  {\rm \cite[Lemma 2.5]{nomstructure}}  \label{lem:T-module}
Let $V$ denote a vector space over $\K$ with finite positive dimension.
Let $(A;\{E_i\}_{i=0}^d;A^*;\{E^*_i\}_{i=0}^d)$ denote a tridiagonal
system on $V$ with eigenvalue sequence $\{\th_i\}_{i=0}^d$ and 
dual eigenvalue sequence $\{\th^*_i\}_{i=0}^d$.
For the sequence $p=(\{\th_i\}_{i=0}^d; \{\th^*_i\}_{i=0}^d)$ let
$T=T(p,\K)$ denote the algebra from Definition {\rm \ref{def:T}}.
Then there exists a unique $T$-module structure on $V$ such that
$a$, $a^*$, $e_i$, $e^*_i$ acts on $V$ as
$A$, $A^*$, $E_i$, $E^*_i$, respectively.
Moreover this $T$-module is irreducible.
\end{lemma}

\medskip

For the rest of this section, let $d$ denote a nonnegative integer and
let $p=(\{\th_i\}_{i=0}^d; \{\th^*_i\}_{i=0}^d)$ denote a sequence
of scalars taken from $\K$ that satisfies conditions (i), (ii) of 
Conjecture {\rm \ref{conj:main}}.
Let $T=T(p,\K)$ denote the corresponding algebra from
Definition \ref{def:T}.
Observe that $e^*_0Te^*_0$ is a $\K$-algebra with multiplicative
identity $e^*_0$.

\medskip

\begin{lemma} {\rm \cite[Theorem~2.6]{nomstructure}} \label{lem:commute}
\samepage
The algebra $e^*_0Te^*_0$ is commutative and generated by
\[
  e^*_0\tau_i(a)e^*_0 \qquad \qquad (1\leq i \leq d).
\]
\end{lemma}

\begin{definition}      \samepage
Let $\{x_i\}_{i=1}^d$ denote mutually commuting indeterminates.
We denote by $\K[x_1,\ldots,x_d]$ the $\K$-algebra consisting of the 
polynomials in $\{x_i\}_{i=1}^d$ that have all coefficients in $\K$.
\end{definition}

\medskip

From Lemma \ref{lem:commute} we immediately obtain the following.

\medskip

\begin{lemma}  {\rm \cite[Corollary 6.3]{NT:muconj}} \label{lem:mu}   \samepage
There exists a surjective $\K$-algebra homomorphism 
\[
    \mu : \K[x_1,\ldots,x_d] \to e^*_0 T e^*_0
\]
that sends $x_i \mapsto e^*_0\tau_i(a)e^*_0$ for $1 \leq i \leq d$.
\end{lemma}

\medskip

The following conjecture was introduced in \cite[Conjecture 6.4]{NT:muconj}.

\medskip

\begin{conjecture}{\rm ($\mu$-conjecture)} \label{conj:mainp}
\samepage
The map $\mu$ from Lemma {\rm \ref{lem:mu}} is an isomorphism.
\end{conjecture}

\medskip

We have shown that the $\mu$-conjecture implies Conjecture \ref{conj:main}
\cite[Theorem 10.1]{NT:muconj}, and that the $\mu$-conjecture holds for
$d\leq 5$ \cite[Theorem 12.1]{NT:muconj}.
In this paper we prove Conjecture \ref{conj:mainp} for the case 
in which $p$ is $q$-Racah.
In our proof we make heavy use of Proposition \ref{prop:ITqRacah}.
Our main result is Theorem \ref{thm:main}.

\section{The main result}

\indent
In this section we obtain our main result, which is Theorem \ref{thm:main}. 
On our way to this result we obtain two other results
Theorem \ref{thm:A} and \ref{thm:FK}, which may be of independent interest.
Throughout this section, let $d$ denote a nonnegative integer and let 
$p=(\{\th_i\}_{i=0}^d; \{\th^*_i\}_{i=0}^d)$ denote a sequence
of scalars taken from $\K$ that satisfies conditions {\rm (i)}, {\rm (ii)}
of Conjecture {\rm \ref{conj:main}}.

\medskip

\begin{theorem}           \label{thm:A}             \samepage
Assume the field $\K$ is infinite and let
$T=T(p,\K)$ denote the $\K$-algebra from Definition {\rm \ref{def:T}}.
Assume that, for every sequence $\{\zeta_i\}_{i=0}^d$ of scalars taken from $\K$
that satisfies condition {\rm (iii)} of Conjecture {\rm \ref{conj:main}},
there exists a sharp tridiagonal system over $\K$ that has
parameter array 
$(\{\th_i\}_{i=0}^d; \{\th^*_i\}_{i=0}^d; \{\zeta_i\}_{i=0}^d)$.
Then the map $\mu : \K[x_1,\ldots,x_d] \to e^*_0Te^*_0$ from 
Lemma {\rm \ref{lem:mu}} is an isomorphism.
\end{theorem}

\begin{proof}
We assume $d \geq 1$; otherwise the result is obvious.
The map $\mu$ is surjective by Lemma \ref{lem:mu}, so it suffices
to show that $\mu$ is injective.
We pick any $f \in \K[x_1,\ldots,x_d]$ such that $\mu(f)=0$, and show $f=0$.
Instead of working directly with $f$, it will be convenient
to work with the product $\psi = fgh$, where
$g = \eta^*_d(\th^*_0)x_d$
and  $h$ is $\eta^*_d(\th^*_0)$ times
\[
  \eta_d(\th_0)+\sum_{i=1}^d \eta_{d-i}(\th_0)x_i.
\]
By construction $\eta^*_d(\th^*_0)\not=0$ so each of
$g,h$ is nonzero. 
To show that $f=0$, we show $\psi=0$ and invoke the fact that 
$\K[x_1,\ldots,x_d]$ is a domain \cite[page 129]{Rot}.
We now show that $\psi =0$.
Since the field $\K$ is infinite it suffices to
show that $\psi(\xi_1,\ldots,\xi_d)=0$ for all
$d$-tuples $(\xi_1,\ldots,\xi_d)$ of scalars taken from $\K$
\cite[Proposition 6.89]{Rot}.
Let $(\xi_1,\ldots,\xi_d)$ be given.
Define
\begin{equation}    \label{eq:defzetai}
 \zeta_i = \xi_i (\th^*_0-\th^*_1)(\th^*_0-\th^*_2)\cdots(\th^*_0-\th^*_i)
   \qquad\qquad (1 \leq i \leq d)
\end{equation}
and $\zeta_0=1$. Observe
\begin{align*}
  g(\xi_1,\ldots,\xi_d) &= \zeta_d,  \\
  h(\xi_1,\ldots,\xi_d) 
    &= \sum_{i=0}^d \eta_{d-i}(\th_0)\eta^*_{d-i}(\th^*_0)\zeta_i.
\end{align*}
First assume $\{\zeta_i\}_{i=0}^d$ does not satisfy condition (iii)
of Conjecture \ref{conj:main}. 
Then either $\zeta_d=0$, in which case
$g(\xi_1,\ldots,\xi_d)=0$, or 
$0 = \sum_{i=0}^d \eta_{d-i}(\th_0)\eta^*_{d-i}(\th^*_0)\zeta_i$, 
in which case $h(\xi_1,\ldots,\xi_d)=0$.
Either way $\psi(\xi_1,\ldots,\xi_d)=0$. 
Next assume $\{\zeta_i\}_{i=0}^d$ does satisfy condition (iii) of 
Conjecture \ref{conj:main}.
By the assumption of the present theorem,
there exists a sharp tridiagonal system
$\Phi=(A;\{E_i\}_{i=0}^d;A^*;\{E^*_i\}_{i=0}^d)$ over $\K$ 
that has parameter array 
$(\{\th_i\}_{i=0}^d; \{\th^*_i\}_{i=0}^d; \{\zeta_i\}_{i=0}^d)$.
Let $V$ denote the vector space underlying $\Phi$.
By Definition \ref{def:split} the following holds on $V$:
\begin{equation}     \label{eq:split2}
 E^*_0\tau_i(A)E^*_0 =
  \frac{\zeta_i E^*_0}
       {(\th^*_0-\th^*_1)(\th^*_0-\th^*_2)\cdots(\th^*_0-\th^*_i)}
  \qquad\qquad (1 \leq i \leq d).
\end{equation}
Consider the $T$-module structure on $V$ from Lemma \ref{lem:T-module}. 
Using \eqref{eq:defzetai}, \eqref{eq:split2} we find that
the following holds on $V$:
\begin{equation}    \label{eq:aux5}
  e^*_0 \tau_i(a) e^*_0 = \xi_i e^*_0  \qquad\qquad   (1 \leq i \leq d).
\end{equation}
Pick an integer $i$ $(1 \leq i \leq d)$.
By \eqref{eq:aux5} the element $e^*_0\tau_i(a)e^*_0$ acts on $e^*_0V$
as $\xi_i$ times the identity map.
Recall that $\mu$ sends $x_i$ to $e^*_0\tau_i(a)e^*_0$,
so $\mu(x_i)$ acts on $e^*_0V$ as $\xi_i$ times the identity map.
By these comments $\mu(f)$ acts on $e^*_0V$ as
$f(\xi_1,\ldots,\xi_d)$ times the identity map.
But $\mu(f)=0$ and $e^*_0V \neq 0$ so
$f(\xi_1,\ldots,\xi_d)=0$,
and therefore $\psi(\xi_1,\ldots,\xi_d)=0$.
We have shown $\psi(\xi_1,\ldots,\xi_d)=0$ for all
$d$-tuples $(\xi_1,\ldots,\xi_d)$ of scalars taken from $\K$,
and therefore $\psi=0$. The result follows.
\end{proof}

\medskip

\begin{theorem}       \label{thm:FK}        \samepage
Let $\F$ denote a field extension of $\K$.
Let the algebras $T_{\K}=T(p,\K)$ and $T_{\F}=T(p,\F)$ be as in Definition 
{\rm \ref{def:T}}.
Let
\begin{align*}
 \mu_{\K} &: \K[x_1,\ldots,x_d] \to e^*_0 T_{\K}e^*_0, \\
 \mu_{\F} &: \F[x_1,\ldots,x_d] \to e^*_0 T_{\F}e^*_0
\end{align*}
denote the maps from Lemma {\rm \ref{lem:mu}} and assume that $\mu_{\F}$ is an isomorphism.
Then $\mu_{\K}$ is an isomorphism.
\end{theorem}

\begin{proof}
The map $\mu_{\K}$ is surjective by Lemma \ref{lem:mu}, 
so it suffices to show that $\mu_{\K}$ is injective.
Since $\K$ is a subfield of $\F$ we may view any $\F$-algebra as a 
$\K$-algebra.
The inclusion map $\K \to \F$ induces an injective $\K$-algebra homomorphism 
$\iota: \K[x_1,\ldots,x_d] \to \F[x_1,\ldots,x_d]$.
In the $\K$-algebra $T_{\F}$ the defining generators 
$a$, $\{e_i\}_{i=0}^d$, $a^*$, $\{e^*_i\}_{i=0}^d$ satisfy the defining
relations for $T_{\K}$. Therefore there exists
a $\K$-algebra homomorphism $N: T_{\K} \to T_{\F}$
that sends each of the $T_{\K}$ generators
$a$, $\{e_i\}_{i=0}^d$, $a^*$, $\{e^*_i\}_{i=0}^d$
to the corresponding generator in $T_{\F}$.
The restriction of $N$ to $e^*_0T_{\K}e^*_0$
is a $\K$-algebra homomorphism $e^*_0T_{\K}e^*_0 \to e^*_0T_{\F}e^*_0$;
we denote this homomorphism by $\nu$. 
By construction the following diagram commutes:
\[
 \begin{CD}
   \K[x_1,\ldots,x_d]  @>{\displaystyle \iota}>> \F[x_1,\ldots,x_d]  \\
   @V{\displaystyle \mu_{\K}}VV  @VV{\displaystyle \mu_{\F}}V   \\
   e^*_0T_{\K}e^*_0  @>{\displaystyle \nu}>> e^*_0T_{\F} e^*_0
 \end{CD}
\]
The maps $\iota$ and $\mu_{\F}$ are injective so their composition
$\mu_{\F} \circ \iota$ is injective. 
But $\mu_{\F} \circ \iota = \nu \circ \mu_{\K}$ so
$\nu \circ \mu_{\K}$ is injective. Therefore $\mu_{\K}$ is injective
and hence an isomorphism.
\end{proof}

\medskip

The following is our main result.

\medskip

\begin{theorem}  \label{thm:main}   \samepage
Let $\K$ denote a field and let $d$ denote a nonnegative integer.
Let $p=(\{\th_i\}_{i=0}^d;$ $\{\th^*_i\}_{i=0}^d)$ 
denote a sequence of scalars taken from $\K$ that is $q$-Racah 
in the sense of Definition {\rm \ref{def:qRacah}}.
Let the $\K$-algebra $T=T(p,\K)$ be as in Definition {\rm \ref{def:T}}.
Then the corresponding map 
$\mu : \K[x_1,\ldots,x_d] \to e^*_0Te^*_0$ 
from Lemma {\rm \ref{lem:mu}} is an isomorphism.
\end{theorem}

\begin{proof}
Abbreviate $\F=\overline{\K}$ for the algebraic closure of $\K$,
and note that $\F$ is infinite. 
Let $T_{\F}=T(p,\F)$ denote the $\F$-algebra from Definition \ref{def:T}
and let $\mu_{\F} : \F[x_1,\ldots,x_d] \to e^*_0 T_{\F}e^*_0$ be the
corresponding map from Lemma \ref{lem:mu}.
By Proposition \ref{prop:ITqRacah}, for every sequence $\{\zeta_i\}_{i=0}^d$
of scalars taken from $\F$ that satisfies condition (iii) of
Conjecture \ref{conj:main}, there exists a sharp tridiagonal system over $\F$
that has parameter array
$(\{\th_i\}_{i=0}^d; \{\th^*_i\}_{i=0}^d; \{\zeta_i\}_{i=0}^d)$.
By this and Theorem \ref{thm:A} the map $\mu_{\F}$ is an isomorphism.
Now $\mu$ is an isomorphism by Theorem \ref{thm:FK}.
\end{proof}

\medskip

We finish with a comment.

\medskip

\begin{lemma}
Proposition {\rm \ref{prop:ITqRacah}} remains true if we drop the assumption 
that $\K$ is algebraically closed.
\end{lemma}

\begin{proof}
Immediate from Theorem \ref{thm:main} and \cite[Theorem 10.1]{NT:muconj}.
\end{proof}

\bigskip

\noindent Kazumasa Nomura \hfil\break
\noindent College of Liberal Arts and Sciences \hfil\break
\noindent Tokyo Medical and Dental University \hfil\break
\noindent Kohnodai, Ichikawa, 272-0827 Japan \hfil\break
\noindent email: {\tt knomura@pop11.odn.ne.jp} \hfil\break

%\noindent Tatsuro Ito \hfil\break
%\noindent Department of Computational Science \hfil\break
%\noindent Faculty of Science \hfil\break
%\noindent Kanazawa University \hfil\break
%\noindent Kakuma-machi \hfil\break
%\noindent Kanazawa 920-1192, Japan \hfil\break
%\noindent email:  {\tt ito@kappa.s.kanazawa-u.ac.jp}

\bigskip

\noindent Paul Terwilliger \hfil\break
\noindent Department of Mathematics \hfil\break
\noindent University of Wisconsin \hfil\break
\noindent 480 Lincoln Drive \hfil\break
\noindent Madison, WI 53706-1388 USA \hfil\break
\noindent email: {\tt terwilli@math.wisc.edu }\hfil\break

\end{document}